\documentclass[12pt,reqno]{amsart}
\usepackage{amsmath,amsthm,amsfonts,amssymb}
\usepackage[french,english]{babel}
\usepackage{tikz}
\usepackage{psfrag,pslatex}
\usepackage{graphicx}
\usepackage{epsfig}

\numberwithin{equation}{section}

\newcommand{\clos}{\operatorname{closure}}

\newcommand{\s} {\sigma}

\def \Z {{\mathbb Z}}
\def \T {{\mathbb T}}
\def \R {{\mathbb R}}

\newtheorem{theorem}{Theorem}[section]

\newtheorem{lemma}[theorem]{Lemma}
\newtheorem{conjecture}[theorem]{Conjecture}
\newtheorem{question}[theorem]{Question}

\newtheorem{obs}[theorem]{Remark}

\theoremstyle{remark}

\begin{document}

\renewcommand{\subjclassname}{\textup{2000} Mathematics Subject Classification}

%\date{\today}

\setcounter{tocdepth}{2}



\renewcommand{\subjclassname}{\textup{2000} Mathematics Subject Classification}

\date{\today}

\title{A non-dynamically coherent example on $\mathbb{T}^3$}

\begin{abstract}

In this paper we give the first example of a non-dynamically coherent partially hyperbolic diffeomorphism with one-dimensional center bundle. The existence of such an example had been an open question since 1975 \cite{bp}. 
\end{abstract}

\thanks{}

\author{F. Rodriguez Hertz}
\address{Department of Mathematics\\ The Pennsylvania State University,
University Park,
State College, PA 16802 .}
\email{hertz@math.psu.edu}

\author{M. A. Rodriguez Hertz}
\address{IMERL-Facultad de Ingenier\'\i a\\ Universidad de la
Rep\'ublica\\ CC 30 Montevideo, Uruguay.}
\email{jana@fing.edu.uy}\urladdr{http://www.fing.edu.uy/$\sim$jana}

\author{R. Ures}
\address{IMERL-Facultad de Ingenier\'\i a\\ Universidad de la
Rep\'ublica\\ CC 30 Montevideo, Uruguay.} \email{ures@fing.edu.uy}
\urladdr{http://www.fing.edu.uy/$\sim$ures}

\maketitle
\section{Introduction}

A diffeomorphism $f$ of a closed manifold $M$ is {\em partially
hyperbolic} if the tangent bundle of $M$, $TM$ splits into three invariant sub-bundles: 
$TM=E^{s}\oplus E^{c}\oplus E^{u}$
such that  all unit vectors
$v^\s\in E^\s_x$ ($\s= s, c, u$) with $x\in M$ satisfy  :
\begin{equation}\label{pointwise.ph}
 \|T_xfv^s\| < \|T_xfv^c\| < \|T_xfv^u\| 
\end{equation}
for some suitable Riemannian metric. The {\em stable bundle} $E^{s}$ must also satisfy
$\|Tf|_{E^s}\| < 1$ and the {\em unstable bundle}, $\|Tf^{-1}|_{E^u}\| < 1$. $E^{c}$ is called {\em center bundle}.\par
It is a well-known fact that the {\em strong} bundles, 
$E^s$ and $E^u$, are uniquely integrable \cite{bp,hps}.  That is, there are invariant foliations ${\mathcal W}^{s}$ and ${\mathcal W}^{u}$ tangent, respectively to the invariant bundles $E^{s}$ and $E^{u}$. \par
However, the integrability of $E^{c}$ is a more delicate matter. There are examples of partially 
hyperbolic diffeomorphisms with non-integrable center bundle. A. Wilkinson observed in \cite{wilk} that there is an Anosov diffeomorphism in a six-nilmanifold, which, when seen as a partially hyperbolic one, has a non-integrable center bundle. One can choose $E^{c}$ to  consist of the weakest part of the Anosov-stable bundle, and the weakest part of the Anosov unstable bundle. These bundles are non-jointly integrable, hence $E^{c}$ is not integrable. This example can be found in S. Smale's survey \cite{sm}, and it is attributed to A. Borel.  Further discussion on these examples can be found, for instance, in 
 \cite{bw_dyncoh}. \par
 $f$ is {\it $cs$-dynamically coherent} if there
exists an $f$-invariant foliation tangent to $E^{s}\oplus E^{c}$. $cu$-dynamical coherence is defined analogously. $f$ is dynamically coherent if it is both $cs$- and $cu$- dynamically
coherent.\par
As seen above, one obstruction to dynamical coherence is
the non-joint integrability of the sub-bundles of $E^{c}$. Is this the only obstruction to dynamical coherence?
What happens, for instance, if the center bundle dimension is one? Would dynamical coherence be then automatic? This question has been open since the 70's. 
Here we show the first example of a non-dynamical coherent partially hyperbolic diffeomorphism with a one-dimensional center bundle.

\begin{theorem}\label{teo.non.dc} There exists a non-void $C^{1}$-open set of partially hyperbolic diffeomorphisms $f:\T^{3}\to\T^{3}$ such that 
\begin{enumerate}
\item $f$ is non-dynamically coherent
\item $f$ admits an invariant $2$-torus tangent to $E^{c}\oplus E^{u}$ 
\end{enumerate}
 
\end{theorem}

The existence of such examples strongly 
contrasts with the following result obtained by M. Brin, D. Burago and S.
Ivanov:

\begin{theorem}[Brin, Burago, Ivanov] \cite{brin_burago_ivanov} All {\em absolutely} partially hyperbolic diffeomorphisms on $\T^{3}$ are dynamically coherent. 
\end{theorem}

Absolute partial hyperbolicity is a more restrictive notion of partial
hyperbolicity, the bounds in its definition are global, contrasting with the usually used pointwise bounds. Namely, $f$ is {\em absolutely partially hyperbolic} if $TM$ admits an invariant splitting into three sub-bundles $TM=E^{s}\oplus E^{c}\oplus E^{u}$ and there are constants $\lambda<1<\mu$ such that all unit vectors $v^{\sigma}\in E_{x}^{\sigma}$, $\sigma=s,c,u$ and $x$ in $M$ satisfy: \par

\begin{equation}\label{absolute.ph}
 \|T_xfv^s\| < \lambda<\|T_xfv^c\| <\mu< \|T_xfv^u\| 
\end{equation}

\vspace*{1em}
Another important issue is the unique integrability. The following question is still open, both for absolutely and pointwise partially hyperbolic diffeomorphisms:
\begin{question}
Assuming $f$ is dynamically coherent, is there a unique invariant foliation tangent to $E^{c}$? 
\end{question}

However, we are able to provide an answer at the local level. If $E$ is a distribution, ${\mathcal W}$ a foliation tangent to $E$ and $W(x)$ is a leaf of ${\mathcal W}$ through the point $x$, $E$ is {\em locally uniquely integrable} at $x$ if any embedded arc through $x$ and tangent to $E$ is contained in $W (x)$.  Local unique integrability implies unique integrability.\par
We prove the following:
\begin{theorem}\label{teo.non.lui} There exists a non-void $C^{1}$ open set of (pointwise) partially hyperbolic diffeomorphisms $f:\T^{3}\to\T^{3}$ satisfying that 
$E^{c}$ is non-locally uniquely integrable.
\end{theorem}

\subsection{Idea of the construction} The idea of the example in Theorem \ref{teo.non.dc} occurred to us while proving that invariant foliations tangent to $E^{c}\oplus E^{u}$ do not have compact leaves \cite{hhu.foliations}. We wanted to prove in fact that there were no compact leaves tangent to $E^{c}\oplus E^{u}$ at all. However, we found that, by perturbing an Anosov map times a Morse-Smale diffeomorphism on the circle, we could obtain a partially  hyperbolic diffeomorphim in $\T^{3}$ with a center-unstable torus $T^{cu}$. \par
It is easy to see that $T^{cu}$ is an attractor. The example was built so that $E^{c}$, and hence $E^{c}\oplus E^{u}$ were uniquely integrable in $\T^{3}\setminus T^{cu}$.This implies that any invariant foliation tangent to $E^{c}\oplus E^{u}$, should contain a compact leaf, which is precluded by \cite{hhu.foliations}.\par
As a matter of fact, we claim that all non-dynamically coherent examples on 3-manifolds have this pattern. That is, they have (at least) an attracting or repelling periodic torus, and trivial dynamics on the rest of the manifold.\par
In fact, we claim that in ``most'' 3-manifolds, {\em all} partially hyperbolic diffeomorphisms are indeed dynamically coherent. More precisely,  
\begin{conjecture}[Hertz-Hertz-Ures (2009)]\label{conj}
If $f:M^3\rightarrow M^3$ is a non-dynamically coherent partially hyperbolic diffeomorphism then it admits a periodic torus tangent to either $E^{c}\oplus E^{u}$ or  $E^{s}\oplus E^{c}$.
\end{conjecture}
Remarkably, a 2-torus like in Conjecture \ref{conj} can occur only in very few 3-manifolds. Indeed, we have the following result:
\begin{theorem}[\cite{anosov_tori}] A partially hyperbolic diffeomorphism on a 3-manifold, admitting a 2-torus tangent to either $E^{s}\oplus E^{u}$, $E^{s}\oplus E^{c}$ or $E^{c}\oplus E^{u}$ can only occur on the following 3-manifold:
\begin{itemize}
\item the 3-torus $\T^{3}$
\item the mapping torus of $-id:\T^{2}\to\T^{2}$
\item the mapping tori of hyperbolic automorphisms on $\T^{2}$ 
\end{itemize} 
\end{theorem}
Also, notice that any of the 2-tori appearing in Conjecture \ref{conj} implies the existence of a periodic torus, which must be either attracting or repelling, since they are transverse to the stable or the unstable foliation. Therefore, if Conjecture \ref{conj} were true, any partially hyperbolic diffeomorphisms for which $\Omega(f)=M$ would be dynamically coherent. \par
Conjecture \ref{conj} has been proven true in the case of the 3-torus by R. Potrie \cite{potrie}, and in the case of  3-solvmanifolds by A. Hammerlindl and Potrie \cite{hp}. But the general case of Conjecture \ref{conj} remains open. What remains to be proven, were this conjecture true, is that {\em all} partially hyperbolic diffeomorphisms in 3-manifolds that are not solvmanifolds, are dynamically coherent.\par
\subsection{Sketch of the proof}

\section{Two examples}\label{section.example}
The non-dynamically coherent example is a 
perturbation of a product of a linear Anosov map on $\T^2$ and a
Morse-Smale map on the circle. The unperturbed map is Axiom A, but not
partially hyperbolic. We will perturb it to obtain partial hyperbolicity, the final diffeomorphism will also be Axiom A. We will only perturb in the stable
direction of the linear map. One can make more complicated examples
by allowing perturbations on the $u$-direction also. Still, our examples leave a $cu$-invariant 
torus. As long as we make an isotopy beginning in this
example and remaining in the partially hyperbolic world, it will have
an invariant $cu$-torus and hence an attractor (see Section \ref{sec.concl}). Hence the Shub type of
construction \cite{shub} of robustly transitive systems does not apply for these examples. 
As we stated in the introduction we believe that in dimension 3 there are no transitive examples of non-dynamically coherent diffeomorphisms. \par
\vspace{0.3cm}
We start by considering the unperturbed diffeomorphism $f_{0}$, which is the product of an Anosov automorphism of the 2-torus by a Morse-Smale map on the circle, that is, 
$f_{0}:\T^2\times S^1\to\T^2\times S^1$ is such that 
$$f_{0}(x,\theta)=(Ax, \psi(\theta))$$
Here $A$ is a hyperbolic
matrix in $SL(2,\Z)$ with eigenvalues $0<\lambda<1$ and $1/\lambda$. And $\psi:S^1\to S^1$ is a North Pole-South Pole map such that
\begin{eqnarray}
&\psi(0)=0, \psi(\frac{1}{2})=\frac{1}{2},& \label{equation.psi.1}\\
&1<\psi'(0)=\mu, \quad \psi'(\frac{1}{2})=\sigma<1,&\label{equation.psi.2}\\
&\sigma<\lambda<1<\mu<\frac{1}{\lambda}.\label{equation.psi.3}&
\end{eqnarray}\par
That is, if $\theta=0$ is the North Pole and $\theta=\frac{1}{2}$ is the South-Pole, then $\psi$ is chosen so that $\psi$ contracts more than the hyperbolic toral automorphism $A$ in the South-Pole, and $\psi$ expands less than $A$ in the North-Pole (see the figure below). \par
Now let $E^{s}_{A}$ be the contracting eigenspace of $A$, and consider a unit vector $e_{s}$ in $E^{s}_{A}$. We shall consider a perturbation $f:\T^{2}\times S^{1}\to \T^{2}\times S^{1}$ of $f_{0}$ of the form
$$f(x,\theta)=(Ax+v(\theta)e_s,\psi(\theta))$$ 

\begin{center}  \includegraphics[height=2cm]{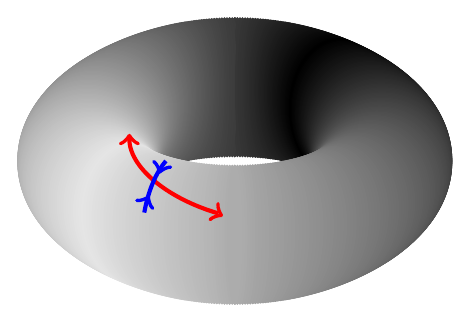}\\ \vspace*{-0.2cm} \includegraphics[height=4cm]{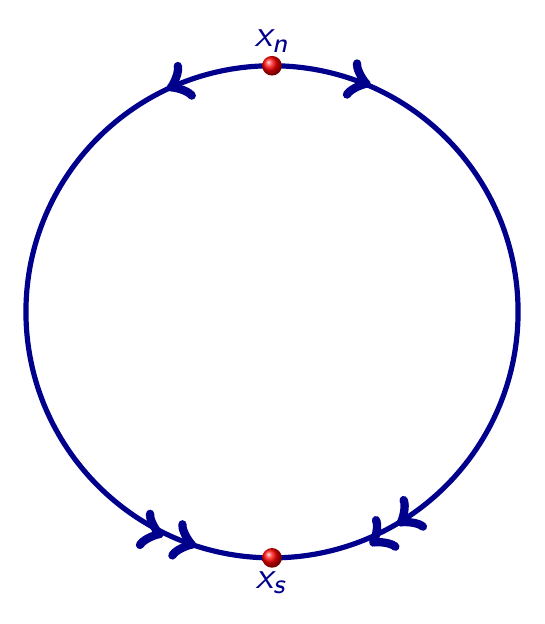}\\ \vspace*{-0.2cm} \includegraphics[height=2cm]{toro3.pdf}
  \end{center}

\par
We will carefully chose the map $v:S^1\to\R$, so that the perturbation $f$ be partially hyperbolic and non-dynamically coherent. Our strategy will be to obtain $f$ such that it be semi-conjugate to $A:\T^{2}\to\T^{2}$, via a semi-conjugacy $h:\T^{3}\to\T^{2}$. The center-stable foliation will be preserved.  For the sake of simplicity, we assume $h$ has the form:
$$h(x,\theta)=x-u(\theta)e_{s}$$
Then, $h\circ f=A\circ h$ yields

\begin{eqnarray*}
h(Ax+v(\theta)e_{s}, \psi(\theta))&=&Ah(x,\theta)\\
Ax+v(\theta)e_{s}-u(\psi(\theta))e_{s}&=&Ax-\lambda u(\theta)e_{\s}\\
v(\theta)-u(\psi(\theta))&=&-\lambda u(\theta)
\end{eqnarray*}

This yields the {\em twisted cohomological equation}:
\begin{equation}\label{twisted.cohomological.equation}
 u(\psi(\theta))-\lambda u(\theta)=v(\theta)\end{equation}

We will consider $v(0)=v(\frac12)=0$. A standard procedure gives that 
% \par
%with $v(1/2)=0$, and let us define
\begin{equation}\label{equation.gamma}
\gamma(\theta)=\frac{1}{\lambda}\sum_{k=1}^\infty \lambda^k
v(\psi^{-k}(\theta))
\end{equation}
is a solution to (\ref{twisted.cohomological.equation}). Since $0<\lambda<1$, $u$ is well defined and continuous,
in fact, it can be seen (see Lemma \ref{difofu}) that $u$ is $C^1$
for $\theta\ne1/2$. 
We also have that 
\begin{equation}\label{equation.beta}
\beta(\theta)=-\frac{1}\lambda\sum_{k=0}^\infty\lambda^{-k}v(\psi^k(\theta))
 \end{equation}
is  a solution to (\ref{twisted.cohomological.equation}). $\beta$ is well defined,
continuous and $C^1$ for $\theta\ne0$, (see Lemma
\ref{difofbeta}).

If we consider 
$$\alpha(\theta)=\gamma(\theta)-\beta(\theta)=\frac{1}{\lambda}\sum_{k\in\Z} \lambda^k
v(\psi^{-k}(\theta))$$
then $\alpha$ is a solution to the equation
\begin{equation}\label{cohomological.equation.2}
 u(\psi(\theta))-\lambda u(\theta)=0
\end{equation}

 that is $C^1$ for $\theta\ne0,1/2$. Observe that if $\alpha\ne0$ then it can not be
bounded.

\begin{lemma}\label{difofu}
$\gamma:S^1\to \R$ is well defined and  continuous. Also $u$ is $C^1$
for $\theta\ne1/2$.
\end{lemma}

\begin{lemma}\label{difofbeta}
$\beta:S^1\setminus \{0\}\to \R$ is well defined and  continuous.
Also $\beta$ is $C^1$ for $\theta\ne0$.
\end{lemma}

Lemmas \ref{difofu} and \ref{difofbeta} are standard exercises and
we leave their proofs to the reader.
\subsection{The partially hyperbolic splitting}
Let us define the partially hyperbolic splitting: First of all, choose $E^u=E^u_A.$
Since the center-stable space $E^s_A\times TS^1$ is preserved by $f$, we
shall search for $E^s_{f}$ and $E^c_{f}$ inside $E^{s}_{A}\times TS^{1}$.\par
Define for $\theta\neq 1/2$,
\begin{equation}\label{E.c}
E^c_{f}(x,\theta)=[(\gamma'(\theta)e^s,1)] 
\end{equation}
 where $[v]$ is the vector
space spanned by $v$. For $\theta=\frac12$, define $E^c_{f}(x,1/2)=E^s_A\times \{0\}$. 
\begin{obs}\label{remark.1}
 Observe that, even though $\gamma'$ is not continuous a priori in $\frac12$, $\psi$ and $v$ can be chosen so that $E^{c}_{f}(x,\theta)$ is indeed continuous on all points of $\T^{3}$, see Lemma \ref{fandv}.
\end{obs}

Analogously, define 
\begin{equation}\label{E.s}
E^s_{f}(x,\theta)=[(\beta'(\theta)e_{s},1)] 
\end{equation}
 if $\theta\neq
0$ and $E^s(x,0)=E^s_A\times\{0\}$.
\begin{obs}\label{remark.2}
The same remark holds for $E^{s}_{f}$, even though $\beta'$ is not continuous a priori in $\theta=0$, $\psi$ and $v$ such that the bundle $E^{s}_{f}(x,\theta)$ is indeed continuous on all points of $\T^{3}$, see Lemma \ref{fandv}
\end{obs}

%%%
The derivative of $f$ is:
$$D_{(x,\theta)}f(v,t)=(Av+tv'(\theta)e_s,t\psi'(\theta)).$$
The bundles $E^{s}_{f}$ and $E^{c}_{f}$ only depend on the variable
$\theta$. Hence, once we have an integral curve of $E^c_{f}$ then
we get the other integral curves by translating it. In a similar way we obtain the stable leaves 
$W^{s}$, once we get one of them.\par

By differentiating equations (\ref{equation.gamma}) and (\ref{equation.beta}), it is not
hard to see that the bundles $E^{c}_{f}$ and $E^{s}_{f}$ are invariant. Moreover, using Lemmas
\ref{difofu} and \ref{difofbeta} we can see that the bundles are
continuous for $\theta\ne1/2,0$.\par
Let us prove that $E^{s}_{f}$ and $E^{c}_{f}$ define a 
splitting with bounded angle. Observe that a measure of the angle between $E^c_{f}$ and $E^s_{f}$ is
given by $\alpha'=\gamma'-\beta'$.\par
%%%
To show that the angle between $E^{s}_{f}$ and $E^{c}_{f}$ is bounded, let us analyze the behavior of $\alpha'$: $S^1\setminus \{0,1/2\}$
consists of two invariant intervals which we identify with $(0,1/2)$
and $(1/2, 1)$. On each interval one has a fundamental domain, which is also an interval. The dynamics is trivial on $(0$ to $1/20,\frac12)$ and on $(\frac12,1)$. \par
Since $\alpha$ is a solution to (\ref{cohomological.equation.2}), by differentiating $\alpha$ we obtain that $\alpha'\circ\psi\cdot \psi'=\lambda\alpha'$. Therefore, if $\alpha'$ is nonzero in a
fundamental domain of, say, $(0,\frac12)$ then it is nonzero everywhere on $(0,\frac12)$.
Analogously for $(\frac12,1)$ Now, on $\theta=0$ and $\theta=1/2$ it is trivial to check that the
angle between $E^s_{f}$ and $E^c_{f}$ is nonzero. By continuity (see Remarks \ref{remark.1} and \ref{remark.2}), in a neighborhood of
$\theta=1/2$ (and of $\theta=0$) there is a complete fundamental domain where this angle is non-zero, and then, $\alpha'$ does
not vanish. By the previous argument $\alpha'$ does not vanish anywhere.
 As a consequence we
have the desired splitting. Finally, the splitting is partially
hyperbolic due to the the fact that it is partially hyperbolic on
$\theta=0$ and $\theta=1/2$.\newline\par

By choosing different functions $v$ and $\psi$, there will be two cases: in one, $\alpha'$ has opposite signs on $(0,\frac12)$ and $(\frac12,1)$; this gives the non-dynamically coherent example. The other case, where $\alpha'$ has the same sign on $(0,\frac12)$ and $(\frac12,1)$ gives a dynamically coherent case, but there $E^{c}_{f}$ is not locally uniquely integrable. This last case is treated in Section \ref{subsection.nonuniquely}.\newline\par
Let us find suitable $v$ and $\psi$:

\begin{lemma}\label{fandv}
There are maps $v$ and $\psi$, such that, if $f(x,\theta)=(Ax+v(\theta)e_{s},\psi(\theta))$, then
$E^c_{f}$ and $E^s_{f}$, as defined in Equations (\ref{E.c}) and (\ref{E.s}) are 
continuous. \par
Namely,  $\lim_{\theta\to 1/2}
|\gamma'(\theta)|=\infty$ and $\lim_{\theta\to 0}
|\beta'(\theta)|=\infty$. Also, $E^u\oplus E^c\oplus E^s$ is a partially
hyperbolic splitting. \par
We can further choose $v$ and $\psi$ so that either $\alpha'$ has different signs on $(0,\frac12)$ and $(\frac12,1)$ (which gives a non-dynamically coherent example), or has the same sign on these intervals (which gives a dynamically coherent example that is locally non-uniquely integrable).
\end{lemma}

Fix some $\epsilon_0>0$ and $0<c_0<1$ such that for $\theta$
with $|\theta-\frac12|\leq\epsilon_0$ or $|\theta|\leq\epsilon_0$ and
for $k\geq 0$ respectively, 
\begin{eqnarray}
\label{compderf}c_0\leq \frac{(\psi^k)'(\theta)}{\sigma^k}\leq
1/c_0\;&\mbox{ and
}&\; c_0\leq \frac{(\psi^{-k})'(\theta)}{\mu^{-k}}\leq 1/c_0 \\
\label{compaproxf}c_0\leq\frac{d(\psi^k(\theta),\frac12)}{\sigma^kd(\theta,\frac12)}\leq
1/c_0\;&\mbox{ and
}&\;c_0\leq\frac{d(\psi^{-k}(\theta),0)}{\mu^{-k}d(\theta,0)}\leq 1/c_0
\end{eqnarray}
Observe that we can take $\epsilon _0$ as close to $0$ as we want
by choosing $c_0$ appropriately. These inequalities follow for any $\psi$ satisfying equations (\ref{equation.psi.2}) and (\ref{equation.psi.3}).\par
We treat the non-dynamically coherent case and the locally non-uniquely integrable case separately: 
\begin{subsection}{The non-dynamically coherent example}\label{sub.nondynco}
\begin{proof}
The non-dynamically coherent case is simpler, just take $v$ such
that
\begin{eqnarray}\label{first}
v'(\theta)<0\mbox{ on }\left(0,\tfrac12\right)&\mbox{and}&v'(\theta)>0\mbox{ on
}\left(\tfrac12,1\right);\\
v''(0)\neq 0&\mbox{and}&v''\left(\tfrac12\right)\neq 0.
\end{eqnarray}
For example $v(\theta)=1+\cos 2\pi\theta$. And take any Morse-Smale map $\psi$ as already chosen.\par
This example readily satisfies the desired properties: 
%That $r'<0$ on $(0,1/2)$ and $r'>0$ on $(1/2,1)$ is
%obvious from the definition of $r$. So w
We only need to show that $E^c_{f}$ and $E^s_{f}$ are continuous, that is
$\lim_{\theta\to 1/2} |\gamma'(\theta)|=\infty$ and $\lim_{\theta\to 0}
|\beta'(\theta)|=\infty$.

%Ones we have done partial hyperbolicity follows from the fact that
%the splitting is obviously p.h. on $\theta=0$ and $\theta=1/2$.

In our case, i.e. $v'(\theta)<0$ for $0<\theta<\frac12$ and
$v'(\theta)>0$ for $\frac12<\theta<1$ with $v''(0)\neq 0$ and
$v''(\frac12)\neq 0$, we have that $v'(0)=0$ and $v'(\frac12)=0$. So for
$\theta$ with $d(\theta,0)\leq \epsilon_0$ or $d(\theta,1/2)\leq
\epsilon_0$ we have respectively,
\begin{eqnarray}\label{compv}
c_0\leq\frac{|v'(\theta)|}{d(\theta,0)}<1/c_0\;&\mbox{ and
}&\;c_0\leq\frac{|v'(\theta)|}{d(\theta,\frac12)}<1/c_0.
\end{eqnarray}

%In the second case (\ref{second}), i.e. $v'(1/2)>0$ we may assume
%also that for $\theta$ with $d(\theta,1/2)\leq \epsilon_0$ we get
%$v'(\theta)\geq \frac{v'(1/2)}{2}>0$.

Let us see that $\gamma'(\theta)\to\infty$ as $\theta\to\frac12$. Given $\theta$ with $d(\theta,\frac12)<\epsilon_0$, let
$N=N(\theta)$ be the largest positive integer such that
$d(f^{-N}(\theta),\frac12)\leq\epsilon_0$. We also assume that $c_0$ is
taken so that
\begin{eqnarray*}
d(f^{-N}(\theta),\tfrac12)&\geq&
c_0d(f^{-1}(f^{-N}(\theta)),f^{-1}(\tfrac12))\\
&=&c_0d(f^{-(N+1)}(\theta),\tfrac12)>c_0\epsilon_0.
\end{eqnarray*}

If $d(\theta,\frac12)\leq \epsilon_0$ and $v'(\theta)>0$, then
\begin{eqnarray*}
\gamma'(\theta)&\geq& \lambda^{(N-1)}v'(f^{-N}(\theta))(f^{-N})'(\theta)
=\lambda^{(N-1)}\frac{v'(f^{-N}(\theta))}{(f^N)'(f^{-N}(\theta))}\\
&\geq& \lambda^{(N-1)}c_0 d(f^{-N}(\theta),1/2)c_0\sigma^{-N}\geq
\lambda^{(N-1)}c_0 c_0\epsilon_0c_0\sigma^{-N}
=\frac{c_0^3\epsilon_0}{\lambda}\left(\frac{\lambda}{\sigma}\right)^N.
\end{eqnarray*}
Finally, if $\theta$ approaches $\frac12$ then $N(\theta)$ tends to
$+\infty$ and hence $\gamma'(\theta)\to +\infty$. Indeed $N(\theta)$ grows essentially as $$\frac{\log d(\theta,\tfrac12)}{\log\sigma}.$$
In case $v'(\theta)<0$, the
computation is the same and we get that $\gamma'(\theta)\to -\infty$ as
$\theta$ approaches $\frac12$. Therefore, for some constants $C>0$
and $\rho=1-\frac{\log\lambda}{\log\sigma}>0$, $\rho<1$, we have
$$|\gamma'(\theta)|\geq Cd(\theta,\tfrac12)^{-\rho}.$$

\begin{figure}
\begin{center}
\psfrag{c}{\tiny{$\gamma^c$}}\psfrag{sa}{\tiny{$E^s_A$}}\psfrag{0}{\tiny{$0$}}
\psfrag{12}{\tiny{$1/2$}}\psfrag{1}{\tiny{$1$}}\psfrag{s}{\tiny{$W^s$}}\psfrag{cs}{\tiny{$E^{cs}$}}
\includegraphics[scale=0.7, width=12cm]{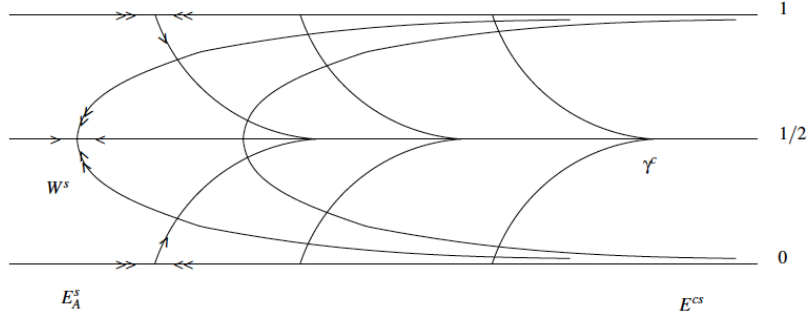}\caption{\label{nonintegrability}The non-dynamically coherent example}
\end{center}
\end{figure}

The same computation gives that $\lim_{\theta\to 0}
|\beta'(\theta)|=\infty$. In this case, we can see that for
$a=1-\frac{\log\lambda}{\log\mu}>2$
$$|\beta'(\theta)|\geq Cd(\theta,0)^{-a}.$$

To see that $E^c$ has no foliation tangent to it, let us explicitly compute the
semiconjugacy. Remember that the semiconjugacy is $h:\T^2\times
S^1\to\T^2$, $h(x,\theta)=x-\gamma(\theta)e^s$. $h$ satisfies $h\circ f=A\circ h$ since
$\gamma$ satisfies the equation \ref{twisted.cohomological.equation}. \par
Now,

$h^{-1}(h(x_0,\theta_0))=\{(x_0+\gamma(\theta)e^s-\gamma(\theta_0)e^s,\theta)\}$
so that
$$h^{-1}(h(x_0,\theta_0))=(h(x_0,\theta_0),0)+\{(\gamma(\theta)e^s,\theta)\}.$$

Recalling that $E^c_{f}(x,\theta)=[(\gamma'(\theta)e^s,1)]$ (Equation (\ref{E.c})), it follows that 
 $h^{-1}(h(x_0,\theta_0))$ are the integral curves
of $E^c$. From the fact that $E^c$ is $C^1$ for $\theta\ne\frac12$ we get that
$E^c$ is uniquely integrable on this domain. On the other hand, on
$\theta=\frac12$ $E^c=E^s_A$, hence the lines which are parallel
to $E^s_A$ inside the torus $\theta=\frac12$ are integral curves of
$E^c$. See Figure \ref{nonintegrability}.
Finally, close to $\theta=\frac12$ we have two situations: either
$\theta\in (0,\frac12)$, and $\gamma'$ is negative, or else 
$\theta\in (\frac12,1)$ and $\gamma'$ is positive. This implies the curves are like in the 
Figure \ref{nonintegrability}, what precludes the integrability of $E^{c}$.\end{proof}

\end{subsection}

\begin{subsection}{A locally non-uniquely integrable example}\label{subsection.nonuniquely} 
\begin{proof}
The construction of a dynamically coherent example which is locally non-uniquely integrable is subtler than the previous case. In our construction, we need to choose $v$ and $\psi$ more care carefully.\par
Firstly, we shall take $\mu$, appearing in Equation (\ref{equation.psi.2}), very close to 1; and $\sigma$, in the same Equation, in such a way that $\frac\sigma\lambda$ is very close to 1 too. Secondly, in order to simplify the calculations, we  choose $\psi$ to be ``symmetric" with respect to $\frac12$. We also choose  $\psi$ in such a way that it is affine with slope $\sigma$ in a neighborhood of $\frac12$   and it is affine with slope $\mu$ in a neighborhood of $1$. Moreover, we can suppose that $\psi$ is  affine outside a fundamental domain $D=(\theta_0,v(\theta_0))$  that depends only on $\mu$ (recall that $\mu$ was already chosen and  $\s<\lambda$).\par

Now we have to define $v$. Let us choose $v$ with odd symmetry (that is, such that $v(-\theta)=-v(\theta)$) with respect to $\frac12$  and in such way $v'(\theta)>0$ for  $\theta \in [\frac12, \theta^*)$ and $v'(\theta)<0$ for  $\theta \in ( \theta^*, 1]$ where $\theta^*$ is a point belonging to $D$.\par

Now let us show that that the bundles $E^c$ and $E^s$ as
defined above are continuous. As it was shown in the first part of the section, the fact that the angle between $E^{s}$ and $E^{c}$ is everywhere non-zero will follow from invariance and continuity.\par

The steps to show that 
$\lim_{\theta\to 1/2} |\gamma'(\theta)|=\infty$ and $\lim_{\theta\to 0}
|\beta'(\theta)|=\infty$, are very 
similar to those in the preceding subsection, let us make the computations for $\gamma'$. 

%Let us prove now that $\lim_{\theta\to 0} |\beta'(\theta)|=\infty$.
%Let us assume that $v'(0)>0$ and take $1/4>\epsilon_1>0$ such that
%$v'(\theta)>\frac{v'(0)}{2}$ for $|\theta|<\epsilon_1$ and take
%$c_0$ as remarked after formula \ref{compderf} with respect to
%$\epsilon_0=1/2-\epsilon_1$. Then we have that choosing
%$N=N(\theta)$ as before but with $\epsilon_1$ we get that
%\begin{eqnarray*}
%|\beta'(\theta)|&\geq&
%\lambda^{-(N+1)}v'(f^N(\theta))(f^N)'(\theta)-1/\lambda\sum_{k=N+1}^\infty\lambda^{-k}|v'(f^k(\theta))|(f^k)'(\theta)\\
%&\geq&
%\lambda^{-(N+1)}\frac{v'(0)}{2(f^{-N})'(f^N(\theta))}-1/\lambda\sum_{k=N+1}^\infty\lambda^{-k}\|v'\|_{C^0}c_0\sigma^k\\
%&\geq&
%\frac{v'(0)\mu^N}{2c_0\lambda^{N+1}}-c_0\|v'\|_{C^0} 1/\lambda\sum_{k=N+1}^\infty\left(\frac{\sigma}{\lambda}\right)^k\\
%\end{eqnarray*}
%The first term goes to infinity and the second one is bounded so we
%get the claim.
%

Suppose that $\theta$ is very close to $\frac12$ and let $N=N(\theta )$
be such that $f'(f^{-i}(\theta))=\s$ for $i=0,\dots,N$. Then
\begin{eqnarray*}
\gamma'(\theta)&=&
\frac{1}{\lambda}\sum_{k=0}^N\lambda^k v'(\psi^{-k}(\theta))(\psi^{-k})'(\theta)+\frac{1}{\lambda}\sum_{k=N+1}^\infty\lambda^k v'(\psi^{-k}(\theta))(\psi^{-k})'(\theta)\\
&\geq& C_1\left(\frac{\lambda}{\sigma}\right)^N\sum_{n=0}^{N} \left(\frac{\sigma}{\lambda}\right)^{N-n} - C_2C_3 \left(\frac{\lambda}{\sigma}\right)^N\sum_{n=N+1}^{\infty}\left(\frac{\lambda}{\mu}\right)^{n-N},\
\end{eqnarray*}

where $C_1$ is a lower bound for $v'$ in its positive region, that is, $v'(\theta)>C_1$ for $\theta \in [\frac12, \theta_0]$, $C_2$ is an upper bound of $|v'|$ in $[\theta_0, 1]$ and $C_3$ is an upper bound of $f'$ in $D$. Observe that the constants $C_i$, $i=1,2,3$, can be taken independent of $\s$ (and of $f$) if $\s$ is close enough to $\lambda$.

Now, choosing $\s$ close enough to $\lambda $ and $N$ sufficiently large (equivalently $\theta $ near enough $\frac12$) we obtain

\begin{eqnarray*}
\gamma'(\theta)&\geq&
\frac1\lambda\left(\frac{\lambda}{\sigma}\right)^N \left(C_1\sum_{n=0}^{N} \left(\frac{\sigma}{\lambda}\right)^{N-n}-C_2C_3\sum_{n=N+1}^{\infty}\left(\frac{\lambda}{\mu}\right)^{n-N}\right)
\\&\geq&
\frac1\lambda\left(\frac{\lambda}{\sigma}\right)^N \left(C_1\frac{1-\left(\frac{\sigma}{\lambda}\right)^{N+1}}{1-(\frac{\sigma}{\lambda})}- C_2C_3\frac1{1-\lambda} \right)\\
\end{eqnarray*}

We can choose $\s$ such that for a large enough $N$, $$C_1\frac{1-(\frac{\sigma}{\lambda})^{N+1}}{1-(\frac{\sigma}{\lambda})}- C_2C_3\frac1{1-\lambda}>0.$$
This implies that $\gamma'$ is positive for $\theta $ close to $\frac12$. The symmetries of $v$ and $\psi$ imply that the same is true for $\theta$ smaller than $\frac12$. The multiplying factor  $\frac1\lambda(\frac{\lambda}{\sigma})^N $ clearly forces  $\gamma'(\theta)\rightarrow+\infty$ as $\theta\rightarrow\frac12$.\newline\par

Finally, as in the preceding subsection, we have a semiconjugacy $h$ and
$$h^{-1}(h(x_0,\theta_0))=(h(x_0,\theta_0),0)+\{(\gamma(\theta)e^s,\theta)\}.$$
Since the sign of $\gamma'$ is the same on both sides of $\theta=\frac12$ this partition is a foliation.

We leave to the reader the proof that $\lim_{\theta\to 0}
|\beta'(\theta)|=\infty$  because it is very similar to the proof for $\gamma'$.
\end{proof}
\end{subsection}

\section{Robustness of the examples and some conclusions}\label{sec.concl}

\subsection{Robustness} In the first part of this section we shall show that the non-dynamical coherence of our examples is a robust property. This is essentially a consequence of the presence of a normally hyperbolic torus tangent to the bundle $E^{cu}$.

\begin{theorem}\label{teo.robdyncon}
There exists an open set $\mathcal{V}\subset \operatorname{Diff}^1(\mathbb{T}^3)$ such that, $\forall f \in \mathcal{V}$, $f$ is a non-dynamically coherent partially hyperbolic  diffeomorphism. Moreover, $f$ presents Reeb-like strips of the strong stable foliation inside the center stable leaves.
\end{theorem}

We also have  an analogous result for the non-locally uniquely integrable dynamically coherent example.

\begin{theorem}\label{teo.robnonuniq}
There is an open set $\mathcal{W}\subset \operatorname{Diff}^1(\mathbb{T}^3)$ such that, $\forall f \in \mathcal{W}$, $f$ is a dynamically coherent partially hyperbolic diffeomorphism but its center bundle is not locally uniquely integrable.
\end{theorem}

Let us first show Theorem \ref{teo.robnonuniq}, since it is easier.

\begin{proof}[Proof of Theorem \ref{teo.robnonuniq}]
 Let $f$ be  one of the examples constructed in Subsection \ref{subsection.nonuniquely}. The torus corresponding to $\theta =\frac12$ is a hyperbolic attractor, it is normally hyperbolic and it is tangent to the bundle $E^{cu}$. Moreover, the center foliation $\mathcal{W}^c$ is also normally hyperbolic with compact leaves and $f$ induces an expansive homeomorphisms in the space of center leaves (conjugated to $A$). In particular, $\mathcal{W}^c$ is plaque expansive. Then, if $g$ is close enough to $f$, $g$ is partially hyperbolic, dynamically coherent (see \cite{hps}) and has a  transitive hyperbolic attractor $T$ that is (diffeomorphic to) a torus and tangent to $E^{cu}_{g}$. Clearly, no center leaf of $g$ can be contained in $T$ while for each point of $T$ there is a complete immersed line tangent to $E^c$ and contained in $T$. This implies that $E^c$ is not locally uniquely integrable and proves the theorem.

\end{proof}

Now, we shall give the proof of Theorem \ref{teo.robdyncon}.
It is a bit more involved than the proof of Theorem \ref{teo.robnonuniq} although the main idea is again that the presence of a center-unstable torus precludes dynamical coherence.

\begin{proof}[Proof of Theorem \ref{teo.robdyncon}]
Let $f$ be  one of the examples constructed in the Subsection \ref{sub.nondynco}.  Observe that $f$ satisfies Axiom A and the strong transversality condition and its nonwandering set consists of the tori corresponding to $\theta=0$ and $\theta=\frac12$. Let $T_1$ be the hyperbolic attractor corresponding to $\theta =\frac12$.  $T_1$ is a normally hyperbolic torus tangent to $E^{cu}$. Let $T_0$ be the hyperbolic repeller corresponding to $\theta =0(=1)$. Observe that even though $T_0$ is not normally hyperbolic, due to its hyperbolicity, it persists under perturbations. This means that a diffeomorphism $g$ close enough to $f$ has a hyperbolic repeller $T_0^g$ homeomorphic to $T_0$. Moreover, since $T_0^g$ is a hyperbolic repeller, it is foliated by its stable manifolds that coincide with the strong stable manifolds of its points. Of course, $g$ also has a hyperbolic attractor $T_0^g$ that is normally hyperbolic and tangent to the bundle $E^{cu}_{g}$.

It is not difficult to see that outside the nonwandering set (the two tori) the center leaves coincide with the intersection of stable manifolds of the attractor and unstable manifolds of the repeller. The continuos variation of these foliations implies that, for a small enough perturbation, the center leaves of $g$ are $C^1$-close to the center leaves of $f$ in the complement of a (small) neighborhood of the two tori. Iterating these curves for the future and the past we obtain the center foliation in the complement of the tori. The length of the center curves obtained in this way is bounded because they are exponentially contracted for the future and the past. The domination of the partially hyperbolic splitting implies that this curves are tangent to $T_1$ (if we add the limit point in this attracting  torus) with the same orientation than the center leaves of $f$. This shows that $g$ is nondynamically coherent. 

Consider also the center-stable foliation $\mathcal{W}^{cs}$ of $f$. Although $f$ is not, $\mathcal{W}^{cs}$ is dynamically coherent. Moreover, $\mathcal{W}^{cs}$ is a normally hyperbolic  foliation by cylinders and it is plaque expansive. Plaque expansiveness is a consequence of the fact that $\mathcal{W}^{cs}$ projects onto the stable foliation of a hyperbolic homeomorphism under the quotient induced by the partition formed by the center circles (other way to obtain plaque expansiveness is to observe that $\mathcal{W}^{cs}$ is $C^1$, even $C^\omega$, see \cite{hps}). Then, thanks to \cite{hps} again, we obtain that $g$ has an invariant center-stable foliation $\mathcal{W}^{cs}_{g}$ whose leaves are cylinders. Any leaf $W^{cs}$ of $\mathcal{W}^{cs}_{g}$ has a stable line  $W^s_0=W^{cs}\cap T_0$. Cutting  $W^{cs}$ along $W^s_0$ we obtain a strip where the strong stable foliation is Reeb-like. The reason of this is that the strong stable foliation of $g$ is very close to the strong stable foliation of $f$ in compact parts. Then, the strong stable foliation of $g$ has the same shape as the strong foliation of $f $ except in a very small neighborhood of $W^s_0$ (the size of the neighborhood depends on closeness of $g$ to $f$). See figure \ref{nonintegrability} in page \pageref{nonintegrability}. Since all strong stable manifolds that are not in $T_0$  intersect $T_1$, and $T_0$ is a repeller we have that these strong manifolds are asymptotic to $W_0^s$. This implies that the restriction of the strong stable foliation to each center stable is Reeb-like and finishes the proof of the Theorem.

\end{proof}

\subsection{Connected components of partially hyperbolic diffeomorphisms}

In this subsection we prove that the set of diffeomorphisms homotopic to $A\times id$ in $\mathbb{T}^2\times\mathbb{S}^1=\mathbb{T}^3$ is not connected.
To be more precise, suppose that $A$ is hyperbolic automorphisms of $\mathbb{T}^2$ and call $\mathcal{PH}_A\subset\operatorname{Diff}^1(\mathbb{T}^3)$ the set of partially hyperbolic diffeomorphisms isotopic to $A\times id$ where $id$ is the identity map of $\mathbb{S}^1$.

\begin{theorem}
$\mathcal{PH}_A$ is disconnected.
\end{theorem}

\begin{proof}
As we have already shown there are diffeomorphisms in $\mathcal{PH}_A$ having an invariant torus tangent to the center-unstable bundle. The set $\mathcal{PH}_A^T$ of such diffeomorphisms is obviously open. Suppose that $f\in \mathcal{PH}_A\cap\clos(\mathcal{PH}_A^T)$ and take $f_n\in \mathcal{PH}_A^T$ converging to $f$. The partial hyperbolicity of $f$ implies that the center unstable $f_n$-invariant tori $T_n$ have basins of attraction of uniform size. Then, the tori $T_n$ converge to an $f$-invariant tori $T$ that is tangent to $E^{cu}(f)$ showing that $f\in \mathcal{PH}_A^T$. This implies that every diffeomorphism in the connected component of one having a center-unstable torus has such a torus too. In particular, the examples of Section \ref{section.example} are not in the component of $A\times id$.
\end{proof}

The same ideas can be used to prove that $\mathcal{PH}_A$ has infinitely many components. This can be achieved  by taking diffeomorphisms with more invariant center-unstable torus. It seems an interesting problem to determine the connected components in function the quantity and  the rotation number of the center unstable (or stable) tori and coherence or incoherence along these tori.

%_______________________________________________________________________

\end{document}